\UseRawInputEncoding
\documentclass[sn-mathphys]{sn-jnl}

\jyear{2021}%

\theoremstyle{thmstyleone}%

\usepackage{amsmath}
\usepackage{amssymb}
\usepackage{longtable}
\usepackage{epstopdf}
\usepackage{multirow,makecell}
\usepackage{hyperref}
\usepackage{graphicx}
\usepackage{algorithm,algorithmicx}

\newtheorem{theorem}{Theorem}
\newtheorem{proposition}[theorem]{Proposition}%

\theoremstyle{thmstyletwo}%
\newtheorem{remark}{Remark}%
\newtheorem{lemma}{Lemma}
\newtheorem{corollary}{Corollary}
\theoremstyle{thmstylethree}%

\begin{document}

\title[A family of BB steplengths]{A family of Barzilai-Borwein steplengths from the viewpoint of scaled total least squares}


\author[1]{\fnm{Shiru} \sur{Li}}\email{lishiru@buaa.edu.cn}

\author[1]{\fnm{Tao} \sur{Zhang}}\email{shuxuekuangwu@buaa.edu.cn}

\author*[1]{\fnm{Yong} \sur{Xia}}\email{yxia@buaa.edu.cn}

\affil[1]{\orgdiv{School of Mathematical Sciences}, \orgname{Beihang University}, \orgaddress{ \city{Beijing}, \postcode{100191}, \country{P. R. China}}}


\abstract{The Barzilai-Borwein (BB) steplengths play great roles in practical gradient methods for solving unconstrained optimization problems. Motivated by the observation that the two well-known BB steplengths correspond to the ordinary and the data least squares, respectively, we introduce a novel family of BB steplengths from the viewpoint of scaled total least squares.  Numerical experiments demonstrate that high performance can be received by a carefully-selected BB steplength in the new family.}




\maketitle

\section{Introduction}
Consider an unconstrained optimization problem:
\begin{eqnarray*}
\min_{x\in\mathbb{R}^n}&~f(x),
\end{eqnarray*}
where $f:\mathbb R^n\rightarrow \mathbb R$ is continuously differentiable.
The classical iterative formulation of the gradient descent method for this problem reads as
\begin{eqnarray*}
	x_{k+1}=x_{k}-\alpha_k g_k,
\end{eqnarray*}
where $g_k$ is the gradient of the objective function $f$ at the $k$-th iteration point $x_k$, and $\alpha_k$ is a steplength with many choices.
The classical steepest descent method dates back to Cauchy \cite{cauchy1847} where $\alpha_k$ is taken as the exact minimizer of $f(x_k-\alpha g_k)$.
The BB steplengths, due to Barzilai and Borwein \cite{barzilai1988two}, are two popular choices of $\alpha_k$. The idea is using the ordinary least squares to solve the simplified quasi-Newton equations (where the quasi-Newton matrix is replaced with $\alpha^{-1}_k I$)
\begin{eqnarray}
&& \alpha_ky_{k-1}=s_{k-1},\label{qnew2}\\
&& \alpha^{-1}_ks_{k-1}=y_{k-1},\label{qnew1}
\end{eqnarray}
respectively,  where $I$ is the identity matrix, $y_{k-1}=g_{k}-g_{k-1}$ and $s_{k-1}=x_{k}-x_{k-1}$. Then the two BB steplengths are given by
\begin{eqnarray}
&&\alpha^{BB1}_{k}=\frac{1}{\arg\min\limits_{\beta} \|\beta s_{k-1}-y_{k-1}\|^2}=\frac{s^T_{k-1} s_{k-1}}{s^T_{k-1} y_{k-1}},\label{b1}\\
&&\alpha^{BB2}_{k}=\arg\min\limits_{\alpha} \|\alpha y_{k-1}-s_{k-1}\|^2=\frac{s^T_{k-1} y_{k-1}}{y^T_{k-1} y_{k-1}},\label{b2}
\end{eqnarray}
where $\|\cdot\|$ is the standard Euclidean norm.

Barzilai and Borwein \cite{barzilai1988two} demonstrated that the Barzilai-Borwein (BB) methods exhibit R-superlinear convergence when applied to the minimization of two-dimensional strongly convex quadratic functions. In their paper, they also showed that, in the worst case, at most four consecutive iteration steps will have a superlinear convergence step. Based on a new analysis, Dai \cite{dai2013new} proved at most three consecutive steps will have a superlinear convergence step for the same problem. For the extension to strongly convex quadratic functions in three dimensions,
Dai \cite{dai2005asymptotic} first  established the R-superlinear convergence. Generally, for any finite-dimensional strongly convex quadratic function, Raydan \cite{raydan1993barzilai} established the convergence and Dai and Liao \cite{dai2002r} established the R-linear convergence rate of the BB methods.

In the nonquadratic case,
Fletcher \cite{fletcher2005} proved that
for a particular strongly convex function suggested by Raydan \cite{raydan1997barzilai},
the sequence generated by the BB methods will diverge.
Raydan \cite{raydan1997barzilai} introduced the non-monotonic line search technique, which was incorporated into the BB method to ensure global convergence for general functions.

The superior numerical performance of the BB methods attracted widespread attention. Many improvements, variants and extensions of the BB methods are proposed. For example, the cyclic BB method \cite {dai2006cyclic}, the stabilized BB method \cite{2019stabilized}, the accelerated BB method \cite{huang2022acceleration}, the projected BB  methods for large-scale box-constrained quadratic programming \cite{dai2005projected}, and the projected alternating BB methods for variational inequalities \cite{qu2021projected}. Recently, Dai et al. \cite{dai2019family} proposed a family of BB methods by considering a convex combination of the long BB steplength $\alpha_k^{BB1}$ and the short BB steplength $\alpha_k^{BB2}$:
\begin{eqnarray}\label{fBB}
\alpha_k(\tau_k)=\tau_k\alpha_k^{BB1}+(1-\tau_k)\alpha_k^{BB2},
\end{eqnarray}
where $\tau_k\in[0,1]$. Some approaches for choosing $\tau_k$ are also suggested \cite{dai2019family}.

In this paper, we introduce a fresh perspective on the BB steplengths by considering various types of least squares.
We first observe that $\alpha^{BB1}$ (\ref{b1}) and $\alpha^{BB2}$ (\ref{b2}) correspond to the ordinary least squares (LS) and the data least squares (DLS), respectively. Building upon this observation, we propose and analyze a novel family of BB steplengths derived from scaled total least squares (STLS), including a newly introduced BB steplength associated with total least squares.

For minimizing $n$-dimensional strictly convex quadratic functions, the BB method with any of the new family of BB steplengths converges R-superlinearly when $n=2$ and R-linearly for any finite $n$.
Furthermore, we extend these newly devised steplengths to general nonquadratic optimization problems, incorporating a non-monotonic line search technique. Numerical results demonstrate the efficiency of the proposed approach.

The remainder of this paper is organized as follows.  Section 2  introduce a novel family of BB steplengths based on scaled total least squares (STLS). In particular, we focus on studying a new BB steplength associated with total least squares. We establish the convergence of the new family of BB methods in Section 3. Section 4 presents numerical results. Conclusions are made in Section 5.

\section{A new family of BB steplengths}
The main contribution of this section is to derive a new family of BB steplengths from the perspective of  scaled total least squares.
\subsection{Scaled least squares}
We briefly review the family of scaled least squares.

Consider the over-determined linear system $Ax\approx b$, where $A\in \mathbb{R}^{m\times n}$ is the data matrix, and $b\in \mathbb{R}^m$ is the observation vector.

If we assume that only the observation vector $b$ contains a noise $r$, solving
\[
\text{LS distance}=\min\limits_{r,x} \{\|r\|^2:Ax=b+r\}=\min\limits_x \|Ax-b\|^2,
\]
yields the well-known ordinary least squares problem.

The less popular data least squares problem \cite{degroat1993data}  corresponds to the case where only the data matrix $A$ is noised by $E$:
\begin{eqnarray*}
\text{DLS distance}=\min\limits_{E,x} \{\|E\|^2_F:(A+E)x=b\}=\min\limits_x \frac{\|Ax-b\|^2}{\|x\|^2},
\end{eqnarray*}
where $E\in\mathbb{R}^{m\times n}$ and $\|\cdot\|_F$ is the Frobenius norm.

In case both $A$ and $b$ are noised, the total least squares problem reads as follows:
\begin{eqnarray*}
\text{TLS distance}=\min\limits_{E,r,x}\{\|E\|^2_F+\|r\|^2:
(A+E)x=b+r\}=\min\limits_{x}\frac{\|Ax-b\|^2}{\|x\|^2+1},
\end{eqnarray*}
which was firstly proposed in \cite{golub1980analysis}.

The above three kinds of least squares can be written in a unified form, the so-called scaled TLS (STLS) problem \cite{paige2002scaled,huffel2002total}, aiming at solving
\begin{eqnarray}\label{STLS2}
\text{STLS distance}=\min\limits_{E,r,x}\{\|E\|^2_F+\|r\|^2:
(A+E)x\gamma=b\gamma+r\},
\end{eqnarray}
where $\gamma>0$ is a parameter.
Based on a proof  similar to that in \cite{beck2006solution}, we can simplify \eqref{STLS2} as follows.
\begin{lemma} \label{thm2}
For any $\gamma>0$, problem \eqref{STLS2} is equivalent to
\[
\min\limits_{x}\frac{\|Ax-b\|^2}{\frac{1}{\gamma^2}+\|x\|^2},
\]
in the sense that both optimal values are equal and share the same minimizer in terms of $x$.
\end{lemma}
\begin{proof}
For any fixed $x$, problem \eqref{STLS2} is a convex quadratic programming problem. The variables $E$ and $r$ are optimal if and only if they satisfy
the KKT conditions:
\[
2E+2\lambda x^T\gamma=0,~2r-2\lambda=0,~(A+E)x\gamma=b\gamma+r.
\]
Solving the above KKT system yields that
\[
r=\frac{\gamma(Ax-b)}{1+\gamma^2\|x\|^2},~
E=-\frac{\gamma^2(Ax-b)x^T}{1+\gamma^2\|x\|^2},
\]
which implies that
\begin{eqnarray*}
\|E\|^2_F+\|r\|^2=\frac{\|Ax-b\|^2}{\frac{1}{\gamma^2}+\|x\|^2},
\end{eqnarray*}
The proof is complete.
\end{proof}

Different $\gamma$ in STLS corresponds to different least squares.
When $\gamma=1$, STLS reduces to TLS. As $\gamma\rightarrow\infty$, STLS approaches to DLS.  As $\gamma\rightarrow 0$, the STLS distance scaled by $1/\gamma^2$ converges to the LS distance.

\subsection{A new family of BB steplengths}
According to Lemma \ref{thm2},  employing STLS to solve the simplified quasi-Newton equations \eqref{qnew2} and \eqref{qnew1} (with $\beta=1/\alpha$) amounts to
\[
\min\limits_\alpha \frac{\|\alpha y_{k-1}-s_{k-1}\|^2}{\frac{1}{\gamma^2}+\alpha^2}~{\rm and}~
\min\limits_\beta \frac{\|\beta s_{k-1}-y_{k-1}\|^2}{\frac{1}{\gamma^2}+\beta^2},\nonumber
\]
respectively.
Their solutions lead to a family of BB steplengths and its ``inverse'' version:
\begin{eqnarray}
\alpha^{BB}_k(\gamma)&=&\frac{ \|s_{k-1}\|^2 -\frac{\|y_{k-1}\|^2}{\gamma^2} +
\sqrt{(\|s_{k-1}\|^2-\frac{\|y_{k-1}\|^2}{\gamma^2})^2+
\frac{4(s_{k-1}^Ty_{k-1})^2}{\gamma^2}}}{2s_{k-1}^Ty_{k-1}},\label{uBB} \\
\alpha^{BB'}_k(\gamma)&=&\frac{2s_{k-1}^Ty_{k-1}}
{ \|y_{k-1}\|^2-\frac{\|s_{k-1}\|^2}{\gamma^2}+
\sqrt{(\frac{\|s_{k-1}\|^2}{\gamma^2} -\|y_{k-1}\|^2)^2+\frac{4(s_{k-1}^Ty_{k-1})^2}{\gamma^2}}},\nonumber
\end{eqnarray}
where $\gamma>0$.

We can observe that both $s_{k-1}$ and $y_{k-1}$ are not exactly provided in the simplified quasi-Newton equations \eqref{qnew2} and \eqref{qnew1}.  This makes the STLS model particularly well-suited, as it accounts for errors in both the data matrix and the observation vector.

Some properties of steplength \eqref{uBB} are summarized in Proposition \ref{thm3}.
The proof of Proposition \ref{thm3}  is available in  Appendix \ref{AA}.
\begin{proposition} \label{thm3}
If $s^T_{k-1} y_{k-1}>0$, for any $\gamma>0$, it holds that
\begin{eqnarray*}
&&\alpha^{BB}_k(\gamma),~\alpha^{BB'}_k(\gamma)\in[\alpha^{BB2}_k,\alpha^{BB1}_k],\\
&&\lim_{\gamma\rightarrow +\infty}\alpha^{BB}_k(\gamma)=
\lim_{\gamma\rightarrow 0}\alpha^{BB'}_k(\gamma)=\alpha^{BB1}_k,\\
&&\lim_{\gamma\rightarrow 0}\alpha^{BB}_k(\gamma)=
\lim_{\gamma\rightarrow +\infty}\alpha^{BB'}_k(\gamma)=\alpha^{BB2}_k.
\end{eqnarray*}
Moreover,
$\alpha^{BB}_k(\gamma)$ and $\alpha^{BB'}_k(\gamma)$ are monotonically increasing and decreasing with respect to $\gamma$, respectively.

If we additionally assume that
$s_{k-1}$ and $y_{k-1}$ are linearly independent, $\alpha^{BB}_k(\gamma)$ and $\alpha^{BB'}_k(\gamma)$ are strictly monotonically increasing and strictly monotonically decreasing, respectively, with respect to the parameter $\gamma$.
\end{proposition}

If $\alpha^{BB1}_k=\alpha^{BB2}_k$, for any $\gamma>0$ and any $\tau_k\in(0,1)$, we have $\alpha^{BB}_k(\gamma)=\alpha_k(\tau_k)$.
If $\alpha^{BB1}_k>\alpha^{BB2}_k$,   Proposition \ref{thm3} implies that for any $\gamma>0$, we have $\alpha^{BB}_k(\gamma)\in(\alpha^{BB2}_k,\alpha^{BB1}_k)$.
Moreover, there exists a unique $\tau_k\in(0,1)$ such that the equation $$\alpha^{BB}_k(\gamma)
=\alpha_k(\tau_k)$$
holds.
Based on \eqref{fBB}, we can solve for $\tau_k$ by setting $\tau_k=\frac{\alpha^{BB}_k(\gamma)-\alpha^{BB2}_k}{\alpha^{BB1}_k-\alpha^{BB2}_k}$. The previous analysis can be summarized with the following corollary.
\begin{corollary} \label{cor}
Suppose that $s^T_{k-1} y_{k-1}>0$ and $\gamma>0$, we have 
\begin{equation}
\alpha^{BB}_k(\gamma)=
\tau_k\alpha_k^{BB1}+(1-\tau_k)\alpha_k^{BB2}=\alpha_k(\tau_k),\label{ubbR}
\end{equation}
where $\tau_k\in(0,1)$ is given by
\begin{eqnarray}\label{tau}
\tau_k=\begin{cases}
\tau,~\forall \tau\in(0,1), &\text{if }\alpha^{BB1}_k=\alpha^{BB2}_k,\\
\\
\frac{\alpha^{BB}_k(\gamma)-\alpha^{BB2}_k}{\alpha^{BB1}_k-\alpha^{BB2}_k},&\text{otherwise}.
\end{cases}
\end{eqnarray}
\end{corollary}

The special case with $\gamma=1$ corresponds to the total least squares problem:
\begin{equation}
\alpha^{BB}_k(1)=\frac{\|s_{k-1}\|^2-\|y_{k-1}\|^2+\sqrt{( \|  y_{k-1}\|^2-\| s_{k-1}\|^2)^2+4(s^T_{k-1} y_{k-1})^2}}{2s^T_{k-1} y_{k-1}}. \label{b3}
\end{equation}
The robustness of $\alpha^{BB}_k(1)$ can be observed from Fig.\ref{fig:1} and the following relation
\[
\alpha^{BB}_k(1)=\min\limits_\gamma\max\{\alpha^{BB}_k(\gamma),~\alpha^{BB'}_k(\gamma)\}.
\]
\begin{figure}[htbp]
	\centering	\includegraphics[width=0.7\textwidth,trim=10 10 10 10,clip]{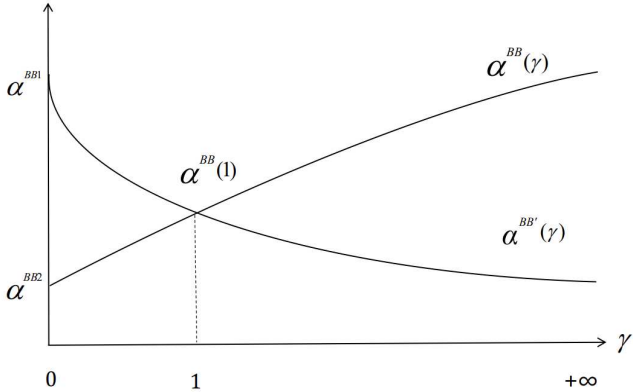}
	\caption{The relation of  $\alpha^{BB}(1)$, $\alpha^{BB}(\gamma)$ and $\alpha^{BB'}(\gamma)$.}\label{fig:1}
\end{figure}
Moreover, $\alpha^{BB}(1)$ keeps a good balance between $\alpha^{BB1}$ and $\alpha^{BB2}$. We list the  observations in the following.
\begin{proposition} \label{thm}
If $s^T_{k-1} y_{k-1}>0$, then we have
\[ \alpha^{BB}_k(1)=\frac{\alpha^{BB1}_k-\frac{1}{\alpha^{BB2}_k}
+\sqrt{\left(\frac{1}{\alpha^{BB2}_k}-\alpha^{BB1}_k\right)^2+4}}{2}.
\]	
Moreover, it holds that
\[\lim\limits_{\alpha^{BB1}_k\ge\alpha^{BB2}_k\to\infty} \frac{\alpha^{BB}_k(1)}{\alpha^{BB1}_k}=1,~\lim\limits_{\alpha^{BB2}_k\le\alpha^{BB1}_k\to 0} \frac{\alpha^{BB}_k(1)}{\alpha^{BB2}_k}=1.	
\]
\end{proposition}
The proof of Proposition \ref{thm}  is provided in Appendix \ref{AB}.

The steplength $\alpha^{BB}_k(1)$ \eqref{b3} was first introduced in \cite{li2021new}, which was the initial version of this paper. Recently, Ferrandi and Hochstenbach  \cite{ferrandi2023homogeneous} re-derived \eqref{b3} using the homogeneous Rayleigh quotient.
Specifically, they considered the minimization of the homogeneous residual:
\begin{eqnarray}\label{hrq}
\min\limits_{\alpha_1^2+\alpha_2^2=1}\|\alpha_1\textbf{u}-\alpha_2A\textbf{u}\|,
\end{eqnarray}
where $A\in\mathbb{R}^{n\times n}$ is a symmetric matrix and $\textbf{u}\in\mathbb{R}^n$ is an approximate eigenvector. The homogeneous Rayleigh quotient is defined as $\alpha=\alpha_1/\alpha_2$.
By substituting $\textbf{u}=s_{k-1}$ and $A=H_k$ with $H_k:=\int^1_0\nabla^2f((1-t)x_{k-1}+tx_k)dt$ in \eqref{hrq}, we obtain
\begin{equation}\label{hrq2}
\min\limits_{\alpha_1^2+\alpha_2^2=1}\|\alpha_1s_{k-1}-\alpha_2y_{k-1}\|,
\end{equation}
and show that its homogeneous Rayleigh quotient is equivalent to  \eqref{b3}. This alternative derivation provides a new perspective on the derivation of \eqref{b3} and clarifies its relationship with the homogeneous Rayleigh quotient.

\section{Convergence analysis}
We analyze the convergence of the BB method with the new steplength \eqref{uBB} for strictly convex quadratic minimization. Extending to minimize nonquadratic problem requires to incorporate a non-monotonic line search technique.
\subsection{Strictly convex quadratic minimization}
We consider minimizing the $n$-dimensional strictly convex quadratic function:
\begin{equation*}
f(x)=\frac{1}{2}x^TAx-b^Tx,
\end{equation*}
where $A\in\mathbb{R}^{n\times n}$ is a symmetric positive definite matrix and $b\in \mathbb{R}^n$.
For this problem, the BB method with the steplength \eqref{fBB} converges R-superlinearly when $n=2$ and R-linearly for any finite $n$ \cite{dai2019family}.

For the case $n=2$, without loss of generality,  we assume that
\begin{eqnarray*}
A=\begin{pmatrix}
1&0\\
0&\lambda
\end{pmatrix},~b=0,
\end{eqnarray*}
where $\lambda> 1$. Then, we define
\begin{eqnarray*}
q_k=\left(\frac{g^{(1)}_k}{g^{(2)}_k}\right)^2,
\theta=\frac{1\pm\sqrt{7}i}{2},~M_k=\log q_k,~\xi_k=M_k+(\theta-1)M_{k-1},
\end{eqnarray*}
where $i$ is the imaginary unit, $g^{(1)}_k$ and $g^{(2)}_k$ are the first and second components of $g_k$, respectively.

\begin{theorem}[\cite{dai2019family}]\label{thm4}
Let $n=2$. Suppose that $\tau_k\in(0,1)$, $g^{(1)}_j\neq 0$ and $g^{(2)}_j\neq 0$ where $j=1,~2$, and $\vert\xi_2\vert\geq 8\log\lambda$ hold, the sequence $\{\|g_k\|\}$ generated by the BB method with the steplength \eqref{fBB} converges to zero R-superlinearly.
\end{theorem}

For the general strictly convex quadratic minimization, we assume
\begin{eqnarray*}
A={\rm diag}\{\lambda_1,\lambda_2,\cdots,\lambda_n\},
\end{eqnarray*}
where $1=\lambda_1\leq\lambda_2\leq\cdots\leq\lambda_n$.

\begin{theorem}[\cite{dai2019family}]\label{thm5}
For the general strictly convex quadratic minimization problem, either the sequence $\{\|g_k\|\}$ generated by the BB method with the steplength \eqref{fBB} converges to zero R-linearly or $g_k=0$ for some finite $k$.
\end{theorem}

As a direct corollary of Theorems \ref{thm4} and \ref{thm5}, we can establish convergence properties of the BB method with the new steplength $\alpha^{BB}_k(\gamma)$ \eqref{uBB}. It is worth noting that the convergence analysis of $\alpha^{BB'}_k(\gamma)$ is similar and will be omitted here.

\begin{theorem}\label{thm6}
Let $\{\|g_k\|\}$ be the sequence generated by the BB method with the steplength \eqref{uBB}, where $\gamma>0$.
For the general strictly convex quadratic minimization problem,  either the sequence $\{\|g_k\|\}$ converges to zero R-linearly, or $g_k=0$ for some finite $k$. For the particular planar case, under the same assumptions as those made in Theorem \ref{thm4}, the sequence $\{\|g_k\|\}$ converges to zero R-superlinearly.
\end{theorem}

The proof of Theorem \ref{thm6} is available in Appendix \ref{AC}.

\subsection{Nonquadratic minimization}
For nonquadratic minimization problem, we combine the steplength \eqref{uBB} with the non-monotonic line search technique \cite{raydan1997barzilai}:
\begin{equation}\label{nls}
f(x_{k+1})\leq\max\limits_{0\leq j\leq M}f(x_{k-j})+\beta g_k^T(x_{k+1}-x_k),
\end{equation}
where $M$ is a nonnegative integer and $\beta>0$. The detailed procedure is presented in Algorithm \ref{algo1}.
\begin{algorithm}
\caption{}\label{algo1}
\begin{algorithmic}[1]
\Require $x_0,~\alpha_0$, integer $M\geq 0,~\beta\in(0,1),~\delta>0,~0<\sigma_1<\sigma_2<1,~0<\eta<1$, $\epsilon>0$ is a precision parameter, and $k=0$.
\Ensure $x_{k+1}$
\State \textbf{if} $\|g_k\|<\epsilon$, stop.
\State \textbf{if} $\alpha_k\leq\eta$ or $\alpha_k\geq1/\eta$ \textbf{then} set $\alpha_k=\delta$.
\If{$f(x_{k}-\alpha_k g_k)\leq\max\limits_{0\leq j\leq \min(k,M)}f(x_{k-j})-\beta\alpha_k g_k^Tg_k$}
        \State $x_{k+1}=x_k-\alpha_kg_k$. 
        \State Calculate $\alpha_{k+1}$ as defined in \eqref{uBB}. Set $k=k+1$ and goto Step 1.
\Else
        \State For $\sigma\in[\sigma_1,\sigma_2]$, set $\alpha_k=\sigma\alpha_k$ and goto Step 3.
\EndIf
\end{algorithmic}
\end{algorithm}

Directly following  \cite[Theorem 2.1]{raydan1997barzilai}, we can establish the global convergence of Algorithm \ref{algo1}.

\begin{theorem}\label{thm8}
Assume that $\Omega_0=\{x:f(x)\leq f(x_0)\}$ is a bounded set. Let $f:R^n\rightarrow R$ be continuously differentiable in  $\Omega_0$. Let $\{x_k\}$ be the sequence generated by Algorithm \ref{algo1}. Then either $g_k=0$ for some finite $k$, or the following properties hold:

{\rm (\romannumeral1)} $\lim\limits_{k\rightarrow\infty}\|g_k\|=0$;

{\rm(\romannumeral2)} no limit point of $\{x_k\}$ is a local maximum of $f$;

{\rm(\romannumeral3)} if the number of stationary points of $f$ in $\Omega_0$ is finite, then the sequence $\{x_k\}$ converges.

\end{theorem}

The proof of Theorem \ref{thm8} is similar to that of \cite[Theorem 2.1]{raydan1997barzilai} and is therefore omitted here.

It is worth noting that the steplength \eqref{uBB} can be combined with other well-known non-monotonic line search techniques, such as those proposed in  \cite{zhang2004nonmonotone}.

\section{Numerical experiments}
We report the numerical performance of our new steplength $\alpha^{BB}_k(\gamma)$ \eqref{uBB} with a carefully-selected constant $\gamma$. In the experiments involving methods with non-monotonic line search, we uniformly set $M=10$, $\beta=0.1$, $\eta=0.001$, $\delta=0.1$ and $\sigma=0.8$.
All experiments were implemented in MATLAB R2017a. All the runs were carried out on a PC with an Intel(R) Core(TM) i7-8665U CPU @1.90GHz 2.11 GHz and 16 GB of RAM.

\subsection{Toward the best constant choice of $\gamma$}
\label{subsec}
In this subsection, we evaluated the performance of various $\gamma$ values using the performance profile \cite{dolan2002benchmarking}.

Consider the following quadratic function taken from  \cite{dai2019family}:
\begin{eqnarray}\label{pro2}
f(x)=\frac{1}{2}x^TQ\cdot {\rm diag}(v_1,\cdots,v_n) \cdot Q^Tx-b^Tx,
\end{eqnarray}
where $b$ is uniformly and randomly generated from $[-10,10]^n$,
\begin{eqnarray*}
Q=(I-2\omega_3\omega_3^T)(I-2\omega_2\omega_2^T)(I-2\omega_1\omega_1^T)
\end{eqnarray*}
with $\omega_1$, $\omega_2$ and $\omega_3$ being unit random vectors, $v_1=1$, $v_n=\kappa$ and $v_j~(j=2,\cdots,n-1)$ is randomly generated from $(1,\kappa)$.

For each $n\in\{100,~1000\}$,
we generated ten instances with different settings $\kappa=10^4,10^5,10^6$. The starting point was set $x_0=(1,1,\cdots,1)^T$. The stop criterion is either that the gradient at the $k$-th iteration satisfies that  $\|g_k\|\le\epsilon\|g_0\|$
with $\epsilon:10^{-6},10^{-9},10^{-12}$  or  the number of iterations exceeds $20000$.

\begin{figure}[htbp]
	\centering	\includegraphics[width=9cm,height=7cm]{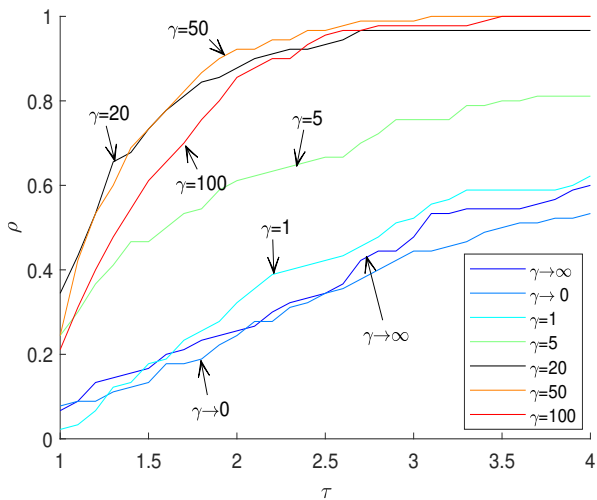}
	\caption{Performance profile of $\alpha^{BB}(\gamma)$ with different $\gamma$ for  $n=100$.}\label{fig:2}
\end{figure}
\begin{figure}[htbp]
	\centering	\includegraphics[width=9cm,height=7cm]{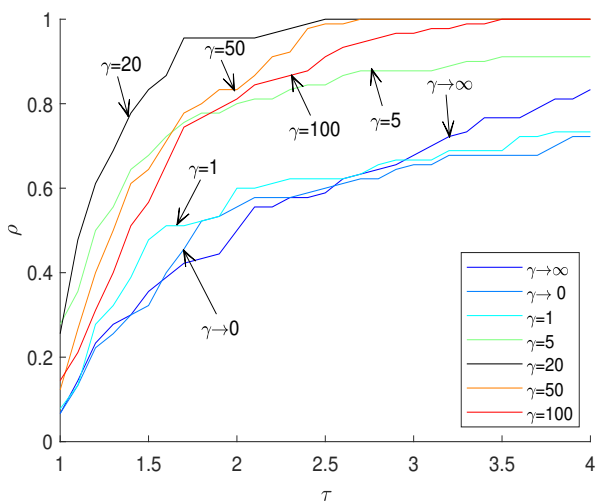}
	\caption{Performance profile of $\alpha^{BB}(\gamma)$ with different $\gamma$ for  $n=1000$.}\label{fig:3}
\end{figure}

Figures \ref{fig:2} and \ref{fig:3} compare the performance of the new steplengths $\alpha^{BB}_k(\gamma)$ with different $\gamma$, where a higher curve represents better performance.
When $n=100$, the best choices of $\gamma$ appear to be $20$ or $50$. The best performance is observed when $n=1000$, $\gamma=20$. In both cases, the steplength \eqref{b3} with $\gamma=1$  is also competitive. The worst two cases are $\gamma\rightarrow0$ and $\gamma\rightarrow +\infty$, which correspond to
 $\alpha^{BB1}_k$ and $\alpha^{BB2}_k$ (see Proposition \ref{thm3}), respectively.

\subsection{Comparing with the state-of-the-art approach}
We test the quadratic function \eqref{pro2}  with seven kinds of settings of $v_j$ summarized in Table \ref{tab:3} taking from \cite{dai2019family}.
We set $n=1000$ and the other settings including the stop criterion are the same as  in Section \ref{subsec}.

\begin{table}[h]\centering
	\begin{center}
\begin{minipage}{160pt}
	\caption{Different settings of $v_j$ for Problem \eqref{pro2}.} \label{tab:3}
		\begin{tabular}{|c|c|} \hline
			&$v_j$      \\
			\hline
			1& $\{v_2,\cdots,v_{n-1}\}\subset(1,\kappa)$\\
\hline
2&\makecell[c]{$\{v_2,\cdots,v_{n/5}\}\subset(1,100)$\\
$\{v_{n/5+1},\cdots,v_{n-1}\}\subset(\frac{\kappa}{2},\kappa)$}\\
\hline
3&\makecell[c]{$\{v_2,\cdots,v_{n/2}\}\subset(1,100)$\\
$\{v_{n/2+1},\cdots,v_{n-1}\}\subset(\frac{\kappa}{2},\kappa)$}\\
\hline
4&\makecell[c]{$\{v_2,\cdots,v_{4n/5}\}\subset(1,100)$\\
$\{v_{4n/5+1},\cdots,v_{n-1}\}\subset(\frac{\kappa}{2},\kappa)$}\\
\hline
5&\makecell[c]{$\{v_2,\cdots,v_{n/5}\}\subset(1,100)$\\
$\{v_{n/5+1},\cdots,v_{4n/5}\}\subset(100,\frac{\kappa}{2})$\\
$\{v_{4n/5+1},\cdots,v_{n-1}\}\subset(\frac{\kappa}{2},\kappa)$}\\
\hline
6&\makecell[c]{$\{v_2,\cdots,v_{10}\}\subset(1,100)$\\
$\{v_{11},\cdots,v_{n-1}\}\subset(\frac{\kappa}{2},\kappa)$}\\
\hline
7&\makecell[c]{$\{v_2,\cdots,v_{n-10}\}\subset(1,100)$\\
$\{v_{n-9},\cdots,v_{n-1}\}\subset(\frac{\kappa}{2},\kappa)$}\\
\hline		
		\end{tabular}
\end{minipage}
	\end{center}
\end{table}

As reported in
\cite{dai2019family}, the following adaptive truncated cyclic scheme (ATC) is
the state-of-the-art dynamic BB steplength:
\begin{eqnarray}\label{ATC}
\alpha^{ATC}_k=\begin{cases}
\alpha^{BB1}_k, &\text{if}\mod(k,m)=0,\\
\widetilde{\alpha}_k,&\text{otherwise},
\end{cases}
\end{eqnarray}
where the cyclic length $m$ is a carefully-selected positive integer and
\begin{eqnarray*}
\widetilde{\alpha}_k=\begin{cases}
\alpha^{BB2}_k, &\text{if}~\alpha_{k-1}\le\alpha^{BB2}_k,\\
\alpha^{BB1}_k, &\text{if}~\alpha_{k-1}\ge\alpha^{BB1}_k,\\
\alpha_{k-1},&\text{otherwise}.
\end{cases}
\end{eqnarray*}

We conducted experiments to numerically compare the new steplengths $\alpha^{BB}(\gamma)$ \eqref{uBB} with $\alpha^{BB}(1)$ \eqref{b3} with $\alpha^{BB1}$ \eqref{b1}, $\alpha^{BB2}$ \eqref{b2}, $\alpha(\tau_k)$ \eqref{fBB}, as well as SDC \cite{de2014efficient}, CBB2 \cite{dai2006cyclic}, ATC \eqref{ATC}, ABBmin1 and ABBmin2 \cite{frassoldati2008new}. The parameter settings in the compared methods are the same as those in \cite{dai2019family}.

We determine the reference values $\tau^*$ and $\gamma^*$ of the optimal parameters in \eqref{fBB} and \eqref{uBB} using the performance profile. As a good approximation to the optimal parameter in \eqref{fBB}, we set $\tau^*=0.9,~0.96,~0.94$ for the fifth, last, and remaining cases in Table \ref{tab:3}, respectively. For the parameter in $\alpha^{BB}(\gamma)$,  we set $\gamma^*=20$ for the first and fifth cases,  and $\gamma^*=2000$ for the other cases
in Table \ref{tab:3}.
We tested ten randomly generated instances for each case and reported the average number of iterations in Table \ref{tab:4}.

Table \ref{tab:4} shows that the minimum number of iterations is primarily determined by ABBmin1, ABBmin2, and $\alpha^{BB}(\gamma)$. Among these, ABBmin1 accounts for the optimal proportion of 28.6 \%, ABBmin2 for 20.6 \%, and $\alpha^{BB}(\gamma)$ for 47.6 \%. Regardless of the tested tolerances, $\alpha^{BB}(\gamma^*)$ is the optimal choice in terms of the total number of iterations for all problems.
Although ABBmin1 and ABBmin2 have a competitive advantage over the best selection in \eqref{uBB}, ABBmin1 requires more storage than $\alpha^{BB}(\gamma)$, while ABBmin2 requires more calculations for matrix-vector multiplication compared to $\alpha^{BB}(\gamma)$. Additionally, ABBmin2 is exclusively applicable to strongly convex quadratic problems and cannot be extended to general problems. Therefore, when a carefully-selected $\gamma$ is employed, $\alpha^{BB}(\gamma)$ is a meaningful choice.

{\footnotesize
\begin{center}
\tabcolsep=0.1em
 \renewcommand{\arraystretch}{1.2}
 		\begin{longtable}{rrrrrrrrrrrrr}
	\caption{
The average number of iterations for various BB steplengths on problem \eqref{pro2} with seven distinct settings in Table \ref{tab:3}.
}\label{tab:4}\\
\hline
& $\kappa$& $\varepsilon$& BB1 & BB2  &SDC&CBB2 &ABBmin1 &ABBmin2 &ATC&$\alpha(\tau^*)$&$\alpha^{BB}(1)$ &$\alpha^{BB}(\gamma^*)$ \\ \hline
 \endfirsthead
\caption{continued}\\
\hline
& $\kappa$& $\varepsilon$& BB1 & BB2  &SDC&CBB2 &ABBmin1 &ABBmin2 &ATC&$\alpha(\tau^*)$&$\alpha^{BB}(1)$ &$\alpha^{BB}(\gamma^*)$\\
\hline
 \endhead
\hline
\endfoot
\hline
\endlastfoot
\multirow{9}*{1}& \multirow{3}*{$10^{4}$}& $10^{-6}$&  443.4 &  478.7 &  425.5 & 1417.9 &  324.2 &  \textbf{263.3} &  476.9 &  462.3 &  461.2 &  374.0   \\
& & $10^{-9}$&  954.6 & 1063.1 &  847.6 & 5832.8 &  580.2 &  \textbf{456.2} &  955.3 & 1041.5 &  953.4 &  843.2 \\
& & $10^{-12}$& 1542.7 & 1691.1 & 1372.9 & 4461.9 &  790.4 &  \textbf{598.6} & 1367.6 & 1550.8 & 1614.1 & 1281.7  \\
 \cline{2-13}
&\multirow{3}*{$10^{5}$} & $10^{-6}$ &  605.8 &  589.0 &  495.6 & 4512.2 &  417.9 &  \textbf{355.6} &  809.5 &  663.7 &  551.1 &  412.2\\
& & $10^{-9}$& 2699.0 & 3313.0 & 1917.6 & 20001.0 &  837.2 &  \textbf{533.6} & 1753.7 & 3276.6 & 2724.1 & 1595.6 \\
& & $10^{-12}$& 4994.8 & 5979.1 & 3872.3 & 20001.0 &  995.6 &  \textbf{918.5} & 3093.2 & 5440.7 & 4736.7 & 2320.1\\
 \cline{2-13}
& \multirow{3}*{$10^{6}$}& $10^{-6}$&  191.5 &  186.3 &  179.6 &  286.6 &  177.8 &  \textbf{162.4} &  497.2 &  200.0 &  193.3 &  195.4\\
& & $10^{-9}$& 6134.5 & 10378.9 & 4715.0 & 20001.0 &  994.9 &  \textbf{509.2} & 2436.7 & 7748.1 & 7344.3 & 1276.6   \\
& & $10^{-12}$& 18188.1 & 20001.0 & 8730.3 & 20001.0 & 1357.2 &  \textbf{687.8} & 4805.8 & 19064.1 & 20001.0 & 2595.8   \\
\hline
\multirow{9}*{2}& \multirow{3}*{$10^{4}$}&$10^{-6}$&  338.2 &  317.3 &  298.6 &  601.0 &  \textbf{232.8} &  258.2 &  334.1 &  321.2 &  334.1 &  277.0  \\
& & $10^{-9}$&  630.6 &  733.5 &  658.2 & 2222.0 &  \textbf{445.2} &  499.3 &  618.2 &  671.3 &  687.6 &  631.7  \\
& & $10^{-12}$&  988.5 & 1040.5 & 1059.5 & 9217.7 &  \textbf{740.9} &  905.7 &  872.7 &  991.9 &  972.4 &  854.7  \\
 \cline{2-13}
&\multirow{3}*{$10^{5}$} & $10^{-6}$&  498.0 &  577.9 &  389.1 & 14117.3 &  283.6 &  496.9 &  469.6 &  550.7 &  559.2 &  \textbf{192.0}  \\
& & $10^{-9}$& 1617.5 & 1718.6 & 1223.5 & 20001.0 &  713.7 & 1202.6 & 1035.4 & 1540.6 & 1600.5 &  \textbf{586.4}  \\
& & $10^{-12}$& 2464.8 & 3056.7 & 2154.0 & 20001.0 & 4248.5 & 18833.7 & 1680.3 & 2759.5 & 3205.7 & \textbf{1022.7} \\
 \cline{2-13}
& \multirow{3}*{$10^{6}$}& $10^{-6}$&  428.8 &  438.3 &  185.3 & 20001.0 &  183.8 &  358.6 &  325.9 &  344.6 &  491.1 &   \textbf{82.1}  \\
& & $10^{-9}$& 3691.4 & 3741.5 & 1596.8 & 20001.0 & 4786.8 & 3954.7 & 1159.5 & 4045.3 & 3564.0 &  \textbf{661.3} \\
& & $10^{-12}$& 6498.3 & 18133.2 & 3421.7 & 20001.0 & 20001.0 & 20001.0 & 2052.8 & 6358.9 & 17741.0 & \textbf{1059.4}\\
 \hline
\multirow{9}*{3}& \multirow{3}*{$10^{4}$}& $10^{-6}$&  350.7 &  385.3 &  346.1 & 2700.1 &  \textbf{250.3} &  311.6 &  354.3 &  367.5 &  373.1 &  302.3  \\
& & $10^{-9}$&  681.8 &  792.0 &  715.9 & 3691.3 &  \textbf{500.2} &  573.0 &  639.3 &  807.1 &  702.4 &  683.2  \\
& &$10^{-12}$& 1015.8 & 1128.0 & 1084.7 & 5041.8 &  \textbf{787.4} & 1222.9 &  938.4 & 1056.7 & 1055.5 &  993.4  \\
 \cline{2-13}
& \multirow{3}*{$10^{5}$}& $10^{-6}$&  639.7 &  591.4 &  448.2 & 20001.0 &  304.3 &  556.4 &  484.4 &  644.9 &  661.7 &  \textbf{206.4} \\
& & $10^{-9}$&  1603.5 & 1791.4 & 1300.5 & 20001.0 &  880.8 & 1379.0 & 1017.7 & 1584.6 & 1646.4 &  \textbf{670.7} \\
& & $10^{-12}$& 2513.7 & 3262.4 & 2227.5 & 20001.0 & 4924.2 & 18026.1 & 1435.3 & 2697.7 & 3020.8 & \textbf{1019.4}  \\
 \cline{2-13}
&\multirow{3}*{$10^{6}$} & $10^{-6}$&  749.9 &  770.2 &  311.4 & 20001.0 &  373.5 &  655.8 &  386.8 &  659.8 &  707.8 &  \textbf{117.9}\\
& & $10^{-9}$& 3525.8 & 4275.7 & 1735.6 & 20001.0 & 4937.7 & 4554.8 & 1291.1 & 3873.1 & 3779.1 &  \textbf{660.4}  \\
& &$10^{-12}$& 6673.5 & 19443.4 & 3208.1 & 20001.0 & 20001.0 & 20001.0 & 1814.4 & 7053.5 & 19701.3 & \textbf{1199.3}  \\
 \hline
\multirow{9}*{4}& \multirow{3}*{$10^{4}$}& $10^{-6}$&  400.6 &  435.1 &  402.3 &  728.2 &  \textbf{271.4} &  346.7 &  400.7 &  431.5 &  403.6 &  345.4\\
& & $10^{-9}$&  752.0 &  769.3 &  720.2 & 2102.3 &  \textbf{497.0} &  611.6 &  670.2 &  785.0 &  714.0 &  727.0  \\
& &$10^{-12}$& 1156.9 & 1119.2 & 1090.6 & 4027.6 &  \textbf{806.6} & 1461.8 &  914.0 & 1138.2 & 1153.7 &  997.2  \\
 \cline{2-13}
& \multirow{3}*{$10^{5}$}& $10^{-6}$&  774.7 &  817.9 &  529.4 & 20001.0 &  344.6 &  653.8 &  524.3 &  728.8 &  829.7 &  \textbf{256.1}\\
& & $10^{-9}$& 1691.2 & 1710.2 & 1300.0 & 20001.0 &  849.7 & 1365.7 & 1124.2 & 1776.1 & 1810.6 &  \textbf{678.5} \\
& & $10^{-12}$& 2751.2 & 3241.9 & 2369.1 & 20001.0 & 4461.5 & 20001.0 & 1540.3 & 2702.9 & 3048.6 & \textbf{1032.9} \\
 \cline{2-13}
& \multirow{3}*{$10^{6}$}& $10^{-6}$&  913.7 & 1059.4 &  396.4 & 20001.0 &  549.5 & 1042.5 &  457.6 &  933.9 &  944.5 &  \textbf{172.1}\\
& & $10^{-9}$& 4378.0 & 4301.6 & 1882.0 & 20001.0 & 6028.9 & 5519.5 & 1344.5 & 4410.1 & 4069.0 &  \textbf{745.7} \\
& & $10^{-12}$ & 6697.0 & 19877.0 & 3455.1 & 20001.0 & 20001.0 & 20001.0 & 2311.9 & 7524.5 & 20001.0 & \textbf{1267.6}\\
 \hline
\multirow{9}*{5}& \multirow{3}*{$10^{4}$}& $10^{-6}$&  603.5 &  606.8 &  571.2 & 1103.6 &  474.1 &  \textbf{430.4} &  574.7 &  638.5 &  571.0 &  690.5\\
& & $10^{-9}$& 1107.9 & 1101.8 & 1044.5 & 3513.1 &  901.4 &  \textbf{828.0} & 1076.5 & 1225.6 & 1131.8 & 1411.5 \\
& & $10^{-12}$& 1738.5 & 1965.4 & 1570.8 & 4138.8 & \textbf{1402.2} & 1432.8 & 1571.9 & 1878.7 & 1790.7 & 2398.4  \\
 \cline{2-13}
& \multirow{3}*{$10^{5}$}&$10^{-6}$& 1244.6 & 1327.4 & \textbf{1139.8} & 14496.1 & 1141.6 & 1238.7 & 1379.3 & 1240.7 & 1229.6 & 1373.5 \\
& & $10^{-9}$& 3411.8 & 4440.5 & 3087.5 & 20001.0 & \textbf{2457.4} & 2880.5 & 3486.8 & 3660.8 & 3507.3 & 6190.5  \\
& &$10^{-12}$ & 6008.2 & 7354.3 & 5415.9 & 20001.0 & \textbf{4235.9} & 19055.3 & 5603.9 & 5730.3 & 6744.9 & 10590.7\\
 \cline{2-13}
& \multirow{3}*{$10^{6}$}&$10^{-6}$ & 1393.1 & 1650.8 & 1326.9 & 18140.9 & 1214.2 & 1820.0 & 2535.9 & 1388.1 & 1761.9 & \textbf{1179.1}\\
& &$10^{-9}$& 11348.8 & 13933.5 & 8507.1 & 20001.0 & \textbf{6691.4} & 11611.6 & 10585.9 & 11923.2 & 11511.6 & 11894.2 \\
& & $10^{-12}$ & 18325.0 & 20001.0 & \textbf{16222.7} & 20001.0 & 18108.0 & 20001.0 & 16876.3 & 19632.5 & 20001.0 & 19877.1 \\
\hline
\multirow{9}*{6}&\multirow{3}*{$10^{4}$} & $10^{-6}$&  268.6 &  313.2 &  229.5 &  893.3 &  174.8 &  \textbf{173.6} &  267.8 &  268.0 &  283.8 &  241.9   \\
& & $10^{-9}$&  646.2 &  556.2 &  491.2 & 1892.0 &  221.1 &  \textbf{219.2} &  526.7 &  610.7 &  479.5 &  567.1\\
& & $10^{-12}$&  962.1 & 1010.2 &  764.1 & 9818.7 &  \textbf{326.4} &  338.1 &  737.2 & 1001.4 &  870.9 &  883.2   \\
 \cline{2-13}
& \multirow{3}*{$10^{5}$}& $10^{-6}$&  394.2 &  380.2 &  323.2 & 14052.3 &  175.9 &  239.9 &  339.2 &  369.7 &  328.0 &  \textbf{120.6}\\
& & $10^{-9}$& 1324.1 & 1382.1 &  843.0 & 20001.0 &  \textbf{338.2} &  413.8 &  713.6 & 1383.0 & 1281.9 &  513.8\\
& &$10^{-12}$& 2157.2 & 2841.8 & 1872.4 & 20001.0 & 1037.5 & 1844.9 & 1319.6 & 2590.1 & 2521.9 &  \textbf{816.4}\\
 \cline{2-13}
& \multirow{3}*{$10^{6}$}& $10^{-6}$&  132.3 &  151.9 &   60.5 & 20001.0 &   25.9 &   54.0 &  166.1 &  131.6 &   96.2 &   \textbf{23.8} \\
& & $10^{-9}$& 2228.4 & 2730.1 & 1157.6 & 20001.0 & 1062.6 & 1357.7 &  868.0 & 2729.6 & 3038.4 &  \textbf{434.7}  \\
& & $10^{-12}$& 4824.2 & 14834.7 & 2460.8 & 20001.0 & 11537.6 & 15616.5 & 1558.3 & 5890.1 & 12922.3 &  \textbf{930.4}  \\
 \hline
\multirow{9}*{7}& \multirow{3}*{$10^{4}$}& $10^{-6}$&  481.4 &  494.5 &  442.8 &  539.6 &  \textbf{308.3} &  389.8 &  454.5 &  489.0 &  454.1 &  389.8\\
& &  $10^{-9}$&  813.9 &  786.0 &  726.3 & 1136.4 &  \textbf{516.7} &  578.5 &  624.0 &  809.4 &  753.2 &  678.4\\
& & $10^{-12}$& 1155.8 & 1194.6 & 1129.0 & 3316.6 &  \textbf{828.8} & 1507.8 &  966.9 & 1300.2 & 1093.9 &  981.9   \\
 \cline{2-13}
&\multirow{3}*{$10^{5}$} & $10^{-6}$&  812.0 &  895.5 &  661.6 & 18043.3 &  421.8 &  653.3 &  565.3 &  732.8 &  831.3 &  \textbf{324.9}\\
& &$10^{-9}$& 1683.7 & 1793.4 & 1229.2 & 20001.0 &  926.8 & 1449.1 & 1051.0 & 1836.9 & 1755.9 &  \textbf{691.8}\\
& &$10^{-12}$ & 2565.1 & 3197.8 & 2204.4 & 20001.0 & 4970.1 & 15631.8 & 1754.8 & 2657.9 & 2980.8 & \textbf{1023.2} \\
 \cline{2-13}
&\multirow{3}*{$10^{6}$} & $10^{-6}$& 1326.1 & 1624.7 &  734.4 & 20001.0 &  576.3 & 1579.9 &  638.1 & 1412.3 & 1545.6 &  \textbf{246.3}\\
& & $10^{-9}$ & 4276.9 & 4010.8 & 1847.2 & 20001.0 & 3747.0 & 8764.4 & 1346.4 & 4077.3 & 4270.5 &  \textbf{726.4} \\
& & $10^{-12}$ & 6204.9 & 19130.2 & 3359.2 & 20001.0 & 19118.9 & 20001.0 & 2097.5 & 6687.7 & 19258.2 & \textbf{1225.1} \\
 \hline	
 \multicolumn{2}{r}{\multirow{3}*{Total}}&$10^{-6}$ & 12990.8 & 14091.8 & 9897.4 & 231639.4 & 8226.6 & 12041.4 & 12442.2 & 12979.6 & 13611.9 & \textbf{7523.3}\\
 &&$10^{-9}$ & 55201.6 & 65323.2 & 37546.5 & 300403.9 & 38914.9 & 49262.0 & 34324.7 & 59815.9 & 57325.5 & \textbf{32868.7}\\
 &&$10^{-12}$& 99426.3 & 169503.5 & 69045.1 & 320037.1 & 140680.7 & 218088.3 & 55313.1 & 105708.3 & 164436.4 & \textbf{54370.6}\\
 \hline

		\end{longtable}

\end{center}
}
\begin{remark}
We only used two fixed choices of $\gamma^*$, namely $20$ or $2000$, in Table \ref{tab:4}. It would achieve better numerical results if we select a more optimal $\gamma^*$ for each individual problem.
\end{remark}

\subsection{The planar Rosenbrock function}
Consider the  classical planar Rosenbrock function \cite{more1981testing}
\begin{eqnarray*}
	f(x)=100(x_2-x^2_1)^2+(1-x_1)^2,~(x_1^0,~x_2^0)=(-1.2,~1).
\end{eqnarray*}
Starting from the initial steplength $\alpha_0=1$,
we performed four independent runs of the BB methods with $\alpha^{BB1}$, $\alpha^{BB2}$, $\alpha^{BB}(1)$ \eqref{b3}, and $\alpha^{BB}(1.5)$, respectively. These runs incorporated the use of a non-monotonic line search, as described in \cite{raydan1997barzilai}.  We set the stop criterion as $\|(x^{k}_1,x^{k}_2)-(x^*_1,x^*_2)\|\le\epsilon$ together with a maximum iteration number $5000$, where $(x^*_1,x^*_2)=(1,1)$ is the minimizer of $f(x)$.
We reported in Table \ref{tab:2} the iteration numbers with different settings of $\epsilon$, where ``--'' stands for the situation that the maximum iteration number is reached.

It is observed that the BB method using $\alpha^{BB}(1.5)$ demonstrates the best performance. Additionally, the efficiency of
$\alpha^{BB}(1)$ is notably higher than that of both $\alpha^{BB1}$ and $\alpha^{BB2}$.
\begin{table}[h]
\begin{center}
\begin{minipage}{170pt}
\caption{Comparison of iteration numbers in minimizing the planar Rosenbrock function.} \label{tab:2}
\begin{tabular}{@{}ccccc@{}}
\toprule
$\epsilon$ &$\alpha^{BB1}$ &$\alpha^{BB2}$ &$\alpha^{BB}(1)$  &$\alpha^{BB}(1.5)$\\
\midrule
$10^{-1}$ & 78& -- & 32&\textbf{29}\\
			$10^{-2}$& 85& --& 38&\textbf{35}\\
			$10^{-4}$& 98& --& 44&\textbf{41}\\
			$10^{-8}$ & 102& --& 46&\textbf{43}  \\
\botrule
\end{tabular}
\end{minipage}
\end{center}
\end{table}

\subsection{More nonquadratic tests}
We conducted tests on several nonquadratic and potentially nonconvex functions $f(x)$, as outlined in \cite{andrei2008}. These tests were initialized using the standard point $x_0$ reported in \cite{andrei2008}. If $f(x_0-g_0/\|g_0\|_\infty)<f(x_0)$, we set $\alpha_0=1/\|g_0\|_\infty$, otherwise $\alpha_0=1/4\|g_0\|_\infty$, where $\|\cdot\|_\infty$ is the infinite norm. The stop criterion is either $\|g_k\|\le10^{-6}\|g_0\|$ or the number of iterations exceeds $10^5$. 

In Table \ref{tab:4}, we present a comparison of various steplength options, which include classical steplengths such as
$\alpha^{BB1}$ and $\alpha^{BB2}$, as well as ABBmin1, ATC, $\alpha^{BB}(1)$, and $\alpha^{BB}(\gamma^*)$. We assess their performance in conjunction with the non-monotonic line search \eqref{nls}. However, we exclude ABBmin2 from consideration as it is not applicable to nonquadratic functions.
The numerical results are presented in Table \ref{tab:5}. Evidently, $\alpha^{BB}(1)$ \eqref{b3} is very efficient in comparison with the others.
Moreover, for each case, the number of iterations with $\gamma=1$ can be further reduced with a better parameter $\gamma^*$ listed in the last column of Table \ref{tab:5}.
\begin{table}[h]
\begin{center}
\tabcolsep=0.1em
\begin{minipage}{\textwidth}
\caption{The number of iterations for different BB steplengths on  nonquadratic test functions.}\label{tab:5}
\begin{tabular*}{\textwidth}{@{\extracolsep{\fill}}lrrrrrrrr@{\extracolsep{\fill}}}
\toprule
function &$n$ &$\alpha^{BB1}$ &$\alpha^{BB2}$ &ABBmin1 &ATC &$\alpha^{BB}(1)$  &$\alpha^{BB}(\gamma^*)$  &$\gamma^*$    \\
\hline
BDQRTIC &5000& 29&28 &28 &27&28&\textbf{24}&8\\
   CRAGGLVY&5000 &436 &56 &60&46&58&\textbf{45}&7\\
CUBE&2 &79 &57 &69&58&55&\textbf{50}&12\\
Diagonal1&20&--&41&50&44&41&\textbf{33}&2\\
Diagonal2&1000&176&135&133&159&174&\textbf{114}&0.02\\
Diagonal3&2000&208&150&206&266&217&\textbf{146}&64\\
EG2&5000&--&--&174&--&--&\textbf{34}&130\\
Extended BD1&4000&18&--&--&18&10&\textbf{10}&1\\
Extended Freudenstein \& Roth &1000&--&32&19&--&21&\textbf{18}&57\\
Extended Rosenbrock&5000&99&47&64&61&39&\textbf{34}&6\\
FLETCHCR&50&551&--&--&990&--&\textbf{191}&1782\\
Hager&5000&73&62&74&71&61&\textbf{55}&18\\
   NONDQUAR &5000&201&88& 124&201&81&\textbf{74}&0.66\\
   Partial Perturbed Quadratic &10000&24&27&23&26&31&\textbf{21}&0.6\\
 POWER &2000&1187&956&688&535&783&\textbf{451}&8815\\
\botrule
\end{tabular*}
\end{minipage}
\end{center}
\end{table}

\begin{remark}
Regarding the $15$ randomly selected functions from \cite{andrei2008}, we found that, without using the non-monotonic line search, there exists a certain value of $\gamma$ for which $\alpha^{BB}(\gamma)$ converges and produces the best experimental results. However, other step length options may not converge for these functions. We should note that adding the non-monotonic line search in the experiments can increase the iteration time.
\end{remark}
\begin{remark}
It is important to note that even with a fixed $\gamma$, $\alpha^{BB}_k(\gamma)$ can be seen as a dynamic convex combination \eqref{ubbR}. Using a fixed $\gamma$ is equivalent to applying a nonlinear transformation to $\tau_k$ in \eqref{fBB}.  Experimental results demonstrated that $\alpha^{BB}_k(\gamma)$ with a fixed $\gamma$ outperforms $\alpha_k(\tau)$ with dynamically selected $\tau_k$. However, the choice of $\gamma$ in $\alpha^{BB}_k(\gamma)$ can also be made dynamically or combined with other design steplength techniques to achieve even better experimental results.
\end{remark}

\section{Conclusions}
We observed that  the two well-known BB steplengths correspond to the ordinary and the data least squares, respectively. Then, based on the scaled total least squares,
we present a new family of BB steplengths with a parameter $\gamma$.
We analyze the convergence of the BB method using the new steplengths.
As $\gamma$ approaches 0 and $+\infty$, the existing BB steplengths are recovered. The setting $\gamma=1$ corresponds to the total least squares. Numerical experiments demonstrate that the steplength with $\gamma=1$ outperforms the two classical BB steplengths in most of the tests. Moreover, the efficiency can be significantly improved with a carefully-selected $\gamma$. It seems that the best parameter $\gamma^*$ is problem dependent. How to (approximately) identify $\gamma^*$ is a critical issue in the future.

\section*{Declarations}
\textbf{Funding}
This research was supported by the National Natural Science Foundation of China under Grant 12171021 and
the Fundamental Research Funds for the Central Universities.
\\
\\
\textbf{Conflict of interest}
The authors declared that they have no conflicts of interest to this work.
\\
\\
\textbf{Availability of data and materials}
 The data and code that support the fndings of this study are available from the corresponding author upon request.


\begin{appendices}

\section{Proof of Proposition \ref{thm3}}\label{AA}
\begin{proof}
For simplicity, we let $a=\|s_{k-1}\|^2$, $b=\|y_{k-1}\|^2$, $c=(s_{k-1}^Ty_{k-1})^2$, $t=\frac{1}{\gamma^2}$. These symbols are used only in Appendix \ref{AA}.
We take the derivative of $\alpha^{BB}_k(t)$ with respect to $t$:
\begin{eqnarray}
[\alpha^{BB}_k(t)]'&=&\frac{1}{2\sqrt{c}}\left[\frac{2c-b(a-bt)}{\sqrt{(a-bt)^2+4ct}}-b\right]
\nonumber\\
&=&\frac{2c-b(a-bt)-b\sqrt{(a-bt)^2+4ct}}{2\sqrt{c}\sqrt{(a-bt)^2+4ct}}
\nonumber\\
&=&\frac{\left[2c-b(a-bt)\right]^2-b^2\left[(a-bt)^2+4ct\right]}{2\sqrt{c}\sqrt{(a-bt)^2+4ct}\left[2c-b(a-bt)+b\sqrt{(a-bt)^2+4ct}\right]}
\nonumber\\
&=&\frac{4c(c-ab)}{2\sqrt{c}\sqrt{(a-bt)^2+4ct}\left[2c-b(a-bt)+b\sqrt{(a-bt)^2+4ct}\right]}
\nonumber\\
&\leq&0, \label{AA0}
\end{eqnarray}
where inequality \eqref{AA0} follows from the assumption $c>0$ and the Cauchy-Schwartz inequality $c\le ab$. Therefore, $\alpha^{BB}_k(\gamma)$ is monotonically increasing with respect to $\gamma$.
 If we additionally assume that $s_{k-1}$ and $y_{k-1}$ are linearly independent, then  $(s^T_{k-1} y_{k-1})^2<\|s_{k-1}\|^2\|y_{k-1}\|^2$, i.e., $c<ab$. That is, inequality \eqref{AA0} strictly holds. It turns out that
  $\alpha^{BB}_k(\gamma)$ is strictly monotonically increasing respect to $\gamma$.

Moreover, by definitions, we can verify that
\begin{eqnarray*}
&&\lim_{\gamma\rightarrow +\infty}\alpha^{BB}_k(\gamma)=
\frac{s^T_{k-1} s_{k-1}}{s^T_{k-1} y_{k-1}}=\alpha^{BB1}_k,\\
&&\lim_{\gamma\rightarrow 0}\alpha^{BB}_k(\gamma)=
\frac{s^T_{k-1} y_{k-1}}{y^T_{k-1} y_{k-1}}=\alpha^{BB2}_k.
\end{eqnarray*}
We will now complete the proof of all properties of  $\alpha^{BB}_k(\gamma)$.
The properties of $\alpha^{BB'}_k(\gamma)$ can be proved similarly, and hence will be omitted.
\end{proof}
\section{Proof of Proposition \ref{thm}}\label{AB}
\begin{proof}
Dividing both sides of \eqref{uBB} by $s^T_{k-1} y_{k-1}$ yields:
\begin{eqnarray}
\alpha^{BB}_k(1)&=&\frac{\frac{\|s_{k-1}\|^2}{s^T_{k-1} y_{k-1}}-\frac{\|y_{k-1}\|^2}{s^T_{k-1} y_{k-1}}+\sqrt{( \frac{\| y_{k-1}\|^2}{s^T_{k-1} y_{k-1}}-\frac{\| s_{k-1}\|^2}{s^T_{k-1} y_{k-1}})^2+4}}{2}\nonumber
\\
&=&\frac{\alpha^{BB1}_k-\frac{1}{\alpha^{BB2}_k}
+\sqrt{\left(\frac{1}{\alpha^{BB2}_k}-\alpha^{BB1}_k\right)^2+4}}{2}\label{A1}.
\end{eqnarray}
To take the limit of \eqref{A1} yields:
$$\lim\limits_{\alpha^{BB2}_k\to\infty} \alpha^{BB}_k(1)=\alpha^{BB1}_k,
~\lim\limits_{\alpha^{BB1}_k\to 0} \alpha^{BB}_k(1)=\alpha^{BB2}_k.$$
The proof is complete.
\end{proof}

\section{Proof of Theorem \ref{thm6}}\label{AC}
\begin{proof}
For the case where $n=2$, the requirements stated in Theorem \ref{thm6} are identical to those of Theorem \ref{thm4}. As proven in \cite{dai2019family}, it is guaranteed that
$g_{k-1}\neq 0$ for any index $k$. Then, by its definition, we have $s_{k-1}\neq 0$.  Due to the positive definiteness of matrix
$A$, we can infer that
\[
s^T_{k-1}y_{k-1}=s^T_{k-1}As_{k-1}>0.
\]
Then, for any $\gamma>0$, by Corollary \ref{cor}, there exists a $\tau_k\in(0,1)$ such that
\[
\alpha^{BB}_k(\gamma)=\alpha_k(\tau_k).
\]

For any finite $n$, either $g_k=0$ for some finite $k$ or $s_{k-1}\neq 0$ from $g_k\neq 0$.

 Then, directly following from Theorems \ref{thm4} and \ref{thm5}  completes the proof.
\end{proof}

\end{appendices}

\bibliography{ref}


\end{document}